\documentclass[11pt,reqno]{amsart}

\usepackage{xspace}
\usepackage{amsmath,amsthm,amsfonts,amssymb,latexsym,mathrsfs,color}
\usepackage{hyperref}

\def\l{\ell}
\def\w{w}
\def\cop{cyclically ordered partition\xspace}
\def\cops{cyclically ordered partitions\xspace}
\def\u{\ensuremath{\mathcal U}\xspace}

\newtheorem{theorem}{Theorem}
\newtheorem{cor}[theorem]{Corollary}
\newtheorem{prop}[theorem]{Proposition}

\newtheorem{remark}[theorem]{Remark}

\newtheorem{lemma}[theorem]{Lemma}

\newcommand{\des}{{\rm des\,}}

\newcommand{\msn}{\mathfrak{S}_n}

\newcommand{\as}{{\rm as\,}}

\newcommand{\Eulerian}[2]{\genfrac{<}{>}{0pt}{}{#1}{#2}}
\newcommand{\Stirling}[2]{\genfrac{\{}{\}}{0pt}{}{#1}{#2}}

\title{Some combinatorial arrays related to the Lotka-Volterra system}

\author[S.-M.~Ma]{Shi-Mei~Ma}
\address{School of Mathematics and Statistics,
        Northeastern University at Qinhuangdao,
         Hebei 066004, P. R. China}
\email{shimeimapapers@gmail.com (S.-M. Ma)}
\author[T.~Mansour]{Toufik Mansour}
\address{Department of Mathematics, University of Haifa, 31905 Haifa, Israel}
\email{toufik@math.haifa.ac.il (T. Mansour)}
\author[D. Callan]{David Callan}
\address{Department of Statistics, University of Wisconsin, Madison, WI 53706}
\email{callan@stat.wisc.edu}

\begin{document}

\maketitle

\begin{abstract}
The purpose of this paper is to investigate the connection between the Lotka-Volterra system and combinatorics.
We study several context-free grammars associated with the Lotka-Volterra system.
Some combinatorial arrays, involving the Stirling numbers of the second kind and Eulerian numbers,
are generated by these context-free grammars. In particular,
we present grammatical characterization of some statistics on cyclically ordered partitions.
\bigskip\\
{\sl Keywords:} Lotka-Volterra system; Context-free grammars; Eulerian numbers; Cyclically ordered partitions
\end{abstract}
\section{Introduction}
One of the most commonly used models of two species predator-prey interaction is the classical
{\it Lotka-Volterra model}:
\begin{equation}\label{Lotka-Volterra-model}
\frac{dx}{dt}=x(a-by),
\frac{dy}{dt}=y(-c+dx),
\end{equation}
where $y(t)$ and $x(t)$ represent, respectively, the predator population and the prey population
as functions of time, and $a,b,c,d$ are positive constants. In general,
an $n$-th order {\it Lotka-Volterra system} takes the form
\begin{equation}\label{Lotka-Volterra-model-2}
\frac{dx_i}{dt}=\lambda_ix_i+x_i\sum_{j=1}^nM_{i,j}x_j, i=1,2,\ldots,n,
\end{equation}
where $\lambda_i,M_{i,j}$ are real real constants.
The differential system~\eqref{Lotka-Volterra-model-2} is ubiquitous and
arises often in mathematical
ecology, dynamical system theory and other branches of
mathematics (see~\cite{Allen07,Chauvet02,Evans99,MiMura78,Shih97,Singer92}).
In this paper, we study several context-free grammars associated with~\eqref{Lotka-Volterra-model-2}.

In his study~\cite{Chen93} of exponential structures in combinatorics,
Chen introduced the grammatical method systematically.
Let $A$ be an alphabet whose letters are regarded as independent commutative indeterminates.
Following Chen, a {\it context-free grammar} $G$ over $A$ is defined as a set
of substitution rules that replace a letter in $A$ by a formal function over $A$.
The formal derivative $D$ is a linear operator defined with respect to a context-free grammar $G$.

Let $[n]=\{1,2,\ldots,n\}$. {\it The Stirling number of the second kind}
$\Stirling{n}{k}$ is the number of ways to partition $[n]$ into $k$ blocks.
It is well known that $S(n,k)=S(n-1,k-1)+kS(n-1,k)$,
$S(0,0)=1$ and $S(n,0)=0$ for $n\geq 1$ (see~\cite[A008277]{Sloane}).
Let $\msn$ be the symmetric group of all permutations of $[n]$.
A {\it descent} of a permutation $\pi\in\msn$ is a position $i$ such that $\pi(i)>\pi(i+1)$.
Denote by $\des(\pi)$ the number of descents of $\pi$.
The {\it Eulerian number} $\Eulerian{n}{k}$ is the number of permutations in $\msn$ with $k-1$ descents, where $1\leq k\leq n$ (see~\cite[A008292]{Sloane}).
Let us now recall two classical results on grammars.
\begin{prop}[{\cite[Eq. 4.8]{Chen93}}]
If $G=\{x\rightarrow xy, y\rightarrow y\}$,
then
\begin{equation*}
D^n(x)=x\sum_{k=1}^n\Stirling{n}{k}y^k\quad\textrm{for $n\ge 1$}.
\end{equation*}
\end{prop}

\begin{prop}[{\cite[Section 2.1]{Dumont96}}]
If $G=\{x\rightarrow xy, y\rightarrow xy\}$,
then
\begin{equation*}
D^n(x)=\sum_{k=1}^{n}\Eulerian{n}{k}x^{k}y^{n-k+1}\quad\textrm{for $n\ge 1$}.
\end{equation*}
\end{prop}

This paper is a continuation of~\cite{Chen93,Dumont96}.
Throughout this paper, arrays are indexed by $n,i$ and $j$.
Call $(a_{n,i,j})$ a combinatorial array if the numbers $a_{n,i,j}$ are nonnegative integers.
For any function $H(x,p,q)$, we denote by $H_y$ the partial derivative of $H$ with respect to $y$, where $y\in\{x,p,q\}$.
In the next section, we present grammatical characterization of some statistics on cyclically ordered partitions.
\section{Relationship to cyclically ordered partitions}\label{sec:02}
Recall that a partition $\pi$ of $[n]$, written $\pi\vdash [n]$, is a collection of disjoint and nonempty subsets
$B_1,B_2,\ldots,B_k$ of $[n]$ such that $\bigcup_{i=1}^kB_i=[n]$, where each $B_i$ ($1\leq i\leq k$) is
called a block of $\pi$. A {\it cyclically ordered partition} of $[n]$ is a partition of $[n]$ whose blocks are endowed with a
cyclic order. We always use a canonical representation for cyclically ordered partitions: a list of blocks in which the first block contains the element 1 and each block is an increasing list. For example, $(123),~(12)(3),~(13)(2),~(1)(23),~(1)(2)(3),~(1)(3)(2)$ are all cyclically ordered partitions of $[3]$. The {\it opener} of a block is its least element. For example, the list of openers of $(13)(2)$ and $(1)(3)(2)$ are respectively  given by $12$ and $132$. In the following, we shall study some statistics on the list of openers.
\subsection{Descent statistic}
\hspace*{\parindent}

Consider the grammar
\begin{equation}\label{Grammar:01}
G=\{x\rightarrow x+xy, y\rightarrow y+xy\}.
\end{equation}

From~\eqref{Grammar:01}, we have
\begin{equation*}
\begin{split}
D(x)&=x+xy,\\
D^2(x)&=x+3xy+xy^2+x^2y,\\
D^3(x)&=x+7xy+6xy^2+xy^3+6x^2y+4x^2y^2+x^3y.
\end{split}
\end{equation*}

For $n\geq 0$, we define
$$D^n(x)=\sum_{i\geq1,j\geq 0}a_{n,i,j}x^iy^j.$$
Since
\begin{align*}
D^{n+1}(x)&=D\left(\sum_{i,j}a_{n,i,j}x^iy^j\right)\\
&=\sum_{i,j}(i+j)a_{n,i,j}x^{i}y^{j}+\sum_{i,j}ia_{n,i,j}x^{i}y^{j+1}+\sum_{i,j}ja_{n,i,j}x^{i+1}y^{j},
\end{align*}
we get
\begin{equation}\label{eq:anij}
a_{n+1,i,j}=(i+j)a_{n,i,j}+ia_{n,i,j-1}+ja_{n,i-1,j}
\end{equation}
for $i,j\geq 1$, with the initial conditions $a_{0,i,j}$ to be $1$ if $(i,j)=(1,0)$, and to be $0$ otherwise.
Clearly, $a_{n,1,0}=1$ and $a_{n,i,0}=0$ for $i\geq 2$.

Define
$$A=A(x,p,q)=\sum_{n,i,j\geq0}a_{n,i,j}\frac{x^{n}}{n!}p^iq^j.$$

We now present the first main result of this paper.
\begin{theorem}\label{mainthm:01}
The generating function $A$ is given by
$$A=\frac{p(p-q)e^x}{p-qe^{(p-q)(e^x-1)}}.$$
Moreover, for all $n,i,j\geq1$,
\begin{equation}\label{Explicit:anij}
a_{n,i,j}=\Stirling{n+1}{i+j}\Eulerian{i+j-1}{i}.
\end{equation}
\end{theorem}
\begin{proof}
By rewriting~\eqref{eq:anij} in terms of generating function $A$, we have
\begin{equation}\label{eq:Axpq}
A_x=p(1+q)A_p+q(1+p)A_q.
\end{equation}
It is routine to check that the generating function
$$\widetilde{A}(x,p,q)=\frac{p(p-q)e^x}{p-qe^{(p-q)(e^x-1)}}$$
satisfies (\ref{eq:Axpq}). Also, this generating function gives $\widetilde{A}(0,p,q)=p$, $\widetilde{A}(x,p,0)=pe^x$ and $\widetilde{A}(x,0,q)=0$ with $q\neq0$. Hence, $A=\widetilde{A}$.
Now let us prove that $a_{n,i,j}=\Stirling{n+1}{i+j}\Eulerian{i+j-1}{i}$. Note that
\begin{align*}
\frac{d}{dx}\sum_{n,i,k\geq0}a_{n,i,k+1-i}\frac{x^{n+1}}{(n+1)!}v^{i+1}w^k
&=v\frac{d}{dx}\sum_{k\geq0}\left(\sum_{n\geq k+1}\Stirling{n+1}{k+1}\frac{x^{n+1}}{(n+1)!}\sum_{i=0}^k\Eulerian{k}{i}v^i\right)w^k\\
&=v\frac{d}{dx}\sum_{k\geq0}\left(\sum_{i=0}^k\Eulerian{k}{i}v^i\right)\frac{(e^x-1)^{k+1}}{(k+1)!}w^k.
\end{align*}
By using the fact that
$$\sum_{k\geq0}\left(\sum_{i=0}^k\Eulerian{k}{i}p^i\right)u^k=\int_0^u\frac{p-1}{p-e^{u'(p-1)}}du'=
\frac{1}{p}(u(p-1)-\ln(e^{u(p-1)}-p)+\ln(1-p)),$$
we obtain that
\begin{align*}
v\frac{d}{dx}\sum_{n,i,k\geq0}a_{n,i,k+1-i}\frac{x^{n+1}}{(n+1)!}v^iw^k
&=\frac{wv(v-1)e^x}{v-e^{(e^x-1)w(v-1)}},
\end{align*}
which implies
\begin{align*}
A(x,vw,w)=\frac{wv(v-1)e^x}{v-e^{(e^x-1)w(v-1)}},
\end{align*}
as required.
\end{proof}

Define $$a_n=\sum_{i\geq1,j\geq 0}a_{n,i,j}.$$
Clearly,
$a_n=\sum_{k= 0}^nk!\Stirling{n+1}{k+1}$.

\begin{prop}
$\Stirling{n}{k} \Eulerian{k-1}{i}$ is the number of cyclically ordered partitions of $[n]$ with $k$ blocks whose list of openers contains $i-1$ descents.
\end{prop}

Proof. To form such a cyclically ordered partition, start with a partition of $[n]$ into $k$ blocks in canonical form, each block increasing and blocks arranged in order of increasing first entries (there are $\Stirling{n}{k}$ choices). The first opener is thus 1. Then leave the first block in place and rearrange the $k-1$ remaining blocks so that their openers, viewed as a list, contain $i-1$ descents (there are $\Eulerian{k-1}{i}$ choices).

We can now conclude the following corollary from the discussion above.
\begin{cor}\label{openers-descents}
For all $n,i,j\geq1$, $a_{n,i,j}$ is the number of cyclically ordered partitions of $[n+1]$ with $i+j$ blocks whose list of openers contains $i-1$ descents.
\end{cor}

\subsection{Peak statistics}
\hspace*{\parindent}

The idea of a peak (resp. valley) in a list of integers $(\w_i)_{i=1}^n$ is an entry that is greater (resp. smaller) than its neighbors. The number of peaks in a permutation is an important combinatorial statistic.
See, e.g.,~\cite{Aguiar04,Dilks09,Francon79,Ma121} and the references therein.
However, the question of whether the first and/or last entry may qualify as a peak (or valley) gives rise to several different definitions. In this paper, we consider only left peaks and right valleys. A {\it left peak index} is an index $i\in[n-1]$ such that $\w_{i-1}<\w_{i}>\w_{i+1}$, where we take $\w_0=0$, and the entry $\w_i$ is a {\it left peak}. Similarly, a {\it right valley}
is an entry $\w_i$ with $i\in[2,n]$ such that $\w_{i-1}>\w_{i}<\w_{i+1}$, where we take $\w_{n+1}=\infty$.
Thus the last entry may be a right valley but not a left peak.  For example, the list $64713258$ has 3 left peaks and 3 right valleys. Clearly, left peaks and right valleys in a list are equinumerous: they alternate with a peak first and a valley last.
Peaks and valleys were considered in~\cite{Francon79}.
The left peak statistic first appeared in~\cite[Definition 3.1]{Aguiar04}.

Let $P(n,k)$ be the number of permutations in $\msn$ with $k$ left peaks.
Let $P_n(x)=\sum_{i\geq 0}P(n,k)x^k$.
It is well known~\cite[A008971]{Sloane} that
\begin{eqnarray*}
  P(x,z) &=&1+ \sum_{n\geq 1}P_n(x)\frac{z^n}{n!}\\
   &=& \frac{\sqrt{1-x}}{\sqrt{1-x}\cosh(z\sqrt{1-x})-\sinh(z\sqrt{1-x})}
\end{eqnarray*}

Let $D$ be the differential operator $\frac{d}{d\theta}$. Set $x=\sec \theta$ and $y=\tan \theta$.
Then
\begin{equation}\label{tansec}
D(x)=xy,D(y)=x^2.
\end{equation}
There is a large literature devoted to the repeated differentiation of
the secant and tangent functions (see \cite{Franssens07,Hoffman99,Ma121} for instance).
As a variation of~\eqref{Grammar:01}, it is natural to
consider the grammar
\begin{equation}\label{Grammar:02}
G=\{x\rightarrow x+xy, y\rightarrow y+x^2\}.
\end{equation}

From~\eqref{Grammar:02}, we have
\begin{equation*}
\begin{split}
D(x)&=x+xy,\\
D^2(x)&=x+3xy+xy^2+x^3,\\
D^3(x)&=x+7xy+6xy^2+xy^3+6x^3+5x^3y.
\end{split}
\end{equation*}

Define $$D^n(x)=\sum_{i\geq1,j\geq 0}b_{n,i,j}x^iy^j.$$

Since
\begin{align*}
D^{n+1}(x)&=D\left(\sum_{i\geq1,j\geq 0}b_{n,i,j}x^iy^j\right)\\
&=\sum_{i,j}(i+j)b_{n,i,j}x^{i}y^{j}+\sum_{i,j}ib_{n,i,j}x^{i}y^{j+1}+
\sum_{i,j}jb_{n,i,j}x^{i+2}y^{j-1},
\end{align*}
we get
\begin{equation}\label{eq:bnij}
b_{n+1,i,j}=(i+j)b_{n,i,j}+ib_{n,i,j-1}+(j+1)b_{n,i-2,j+1}
\end{equation}
for $i\geq 1$ and $j\geq 0$, with the initial conditions $b_{0,i,j}$ to be $1$ if $(i,j)=(1,0)$, and to be $0$ otherwise.
Clearly, $b_{n,1,0}=1$ for $n\geq 1$.

Define
$$B=B(x,p,q)=\sum_{n,i,j\geq0}b_{n,i,j}p^iq^j\frac{x^n}{n!}.$$

We now present the second main result of this paper.
\begin{theorem}\label{mainthm:02}
The generating function $B$ is given by
$$B(x,p,q)=\frac{p\sqrt{q^2-p^2}e^x}{\sqrt{q^2-p^2}\cosh(\sqrt{q^2-p^2}(e^x-1))-q\sinh(\sqrt{q^2-p^2}(e^x-1))}.$$
Moreover, for all $n,i,j\geq1$,
\begin{equation}\label{Explicit:bnij}
b_{n,2i-1,j}=\Stirling{n+1}{2i-1+j}P(2i-2+j,i-1).
\end{equation}
\end{theorem}
\begin{proof}
The recurrence~\eqref{eq:bnij} can be written as
\begin{equation}\label{eq:Bxpq}
B_x=p(1+q)B_p+(p^2+q)B_q.
\end{equation}
It is routine to check that the generating function
$$\widetilde{B}=\widetilde{B}(x,p,q)=\frac{p\sqrt{q^2-p^2}e^x}{\sqrt{q^2-p^2}\cosh(\sqrt{q^2-p^2}(e^x-1))-q\sinh(\sqrt{q^2-p^2}(e^x-1))}$$
satisfies~\eqref{eq:Bxpq}). Also, this generating function gives $\widetilde{B}(0,p,q)=p$ and $\widetilde{B}(x,0,q)=0$. Hence, $B=\widetilde{B}$.

It follows from~\eqref{eq:bnij} that $b_{n,2i,j}=0$ for all $(i,j)\neq (0,0)$. Now let us prove that $$b_{n,2i-1,j}=\Stirling{n+1}{2i-1+j}P(2i-2+j,i-1).$$

Note that
\begin{align*}
\sum_{n,i,j\geq0}b_{n,i,j+1-2i}p^iq^j\frac{x^n}{n!}&=\sum_{n\geq0,i,j\geq1}b_{n,2i-1,j+1-2i}p^iq^j\frac{x^n}{n!}=p\sum_{n\geq0,j\geq1}\Stirling{n+1}{j}P_{j-1}(p)q^j\frac{x^n}{n!}\\
&=pe^x\sum_{j\geq1}\frac{(e^x-1)^{j-1}}{(j-1)!}P_{j-1}(p)q^j=pqe^xP(p,q(e^x-1)),
\end{align*}
Hence,
\begin{equation*}
\sum_{n,i,j\geq0}b_{n,i,j}p^iq^j\frac{x^n}{n!}=pe^xP(p^2/q^2,q(e^x-1))=B(x,p,q),
\end{equation*}
as required.
\end{proof}

Let $b_n=\sum_{i\geq1,j\geq 0}b_{n,i,j}$.
It follows from~\eqref{Explicit:cnij} that
$b_n=a_n$. In the following discussion,
we shall present a combinatorial interpretation for $b_{n,i,j}$.

\begin{lemma} \label{insert}
Suppose that $(\w_i)_{i=1}^k$ is a list of distinct integers containing $\ell$ right valleys and that $\w_1=1$. Then, among the $k$ ways to insert a new entry $m>\max(\w_i)$ into the list in a noninitial position, $2\ell+1$ of them will not change the number of right valleys and $k-(2\ell+1)$ will increase it by 1.
\end{lemma}

Proof. As observed above, peaks and valleys alternate, a peak occurring first, and a valley occurring last. Thus there are $\l$ peaks. If $m$ is inserted immediately before or after a peak or at the very end, the number of valleys is unchanged, otherwise it is increased by 1.

\begin{prop}
The number $u_{n,k,\l}$ of \cops on $[n]$ with $k$ blocks and $\l$ right valleys in the list of openers satisfies the recurrence
\begin{equation}\label{eq:recurrenceu}
u_{n,k,\l} = k u_{n-1,k,\l} + (2\l+1)u_{n-1,k-1,\l} + (k-2\l)u_{n-1,k-1,\l-1}
\end{equation}
for $n\ge 2,\ \l \ge 0,  \ 2\l+1 \le k \le n$.
\end{prop}
Proof. Each \cop of size $n$ is obtained  by inserting $n$ into one of size $n-1$, either as the last entry in an existing block or as a new singleton block.
Let $\u_{n,k,\l}$ denote the set of \cops counted by $u_{n,k,\l}$. To obtain an element of $\u_{n,k,\l}$ we can insert $n$ into any existing block of an element of  $\u_{n-1,k,\l}$ (this gives \,$k u_{n-1,k,\l}$ choices\,), or insert $n$ as a singleton block into an element of $\u_{n-1,k-1,\l}$ so that the number of right valleys is unchanged (this gives\,$(2\l+1)u_{n-1,k-1,\l}$ choices\,), or insert $n$ as a singleton block into an element of $\u_{n-1,k-1,\l-1}$ so that the number of right valleys is increased by 1 (this gives \,$(k-2\l)u_{n-1,k-1,\l-1}$ choices\,). The last two counts of choices follow from Lemma~\ref{insert}.

\begin{cor}
For all $n,i,j\geq1$, $b_{n,i,j}$ is the number of cyclically ordered partitions on $[n+1]$ with $i+j$ blocks and $\frac{i-1}{2}$ right valleys (equivalently, $\frac{i-1}{2}$ left peaks) in the list of openers.
\end{cor}

Proof. Comparing recurrence relations~\eqref{eq:bnij} and~\eqref{eq:recurrenceu},
we see that $b_{n,i,j}=u_{n+1,i+j,(i-1)/2}$.

\begin{remark}
A cyclically ordered partition of size $n$ with $k$ blocks and $\ell$ right valleys in the list of openers is obtained by selecting a partition of $[n]$ with $k$ blocks in $\Stirling{n}{k}$ ways, and then arranging the blocks suitably, in $P(k,\ell)$ ways. Hence $u_{n,k,\l}=
\Stirling{n}{k}P(k,\ell)$ and we get a combinatorial proof that $c_{n,2i-1,j}=\Stirling{n+1}{2i-1+j}P(2i-2+j,i-1)$.
\end{remark}

\subsection{The longest alternating subsequences}
\hspace*{\parindent}

Let $\pi=\pi(1)\pi(2)\cdots \pi(n)\in\msn$.
An {\it alternating subsequence} of $\pi$ is a subsequence $\pi({i_1})\cdots \pi({i_k})$ satisfying
$$\pi({i_1})>\pi({i_2})<\pi({i_3})>\cdots \pi({i_k}).$$
The study of the distribution of the length of the
longest alternating subsequences of permutations was recently initiated by
Stanley~\cite{Sta07,Sta08}.

Denote by $\as(\pi)$ the length of the longest alternating subsequence of $\pi$.
Let
$$a_k(n)=\#\{\pi\in\msn:\as(\pi)=k\},$$
and let
$L_n(x)=\sum_{k=1}^na_k(n)x^k$.
Define
\begin{equation*}\label{Txz-def}
L(x,z)=\sum_{n\geq 0}L_n(x)\frac{z^n}{n!}.
\end{equation*}
Stanley~\cite[Theorem 2.3]{Sta08} obtained the following closed-form formula:
\begin{equation*}\label{Stanley}
L(x,z)=(1-x)\frac{1+\rho+2xe^{\rho z}+(1-\rho)e^{2\rho z}}{1+\rho-x^2+(1-\rho-x^2)e^{2\rho z}},
\end{equation*}
where $\rho=\sqrt{1-x^2}$. Moreover, it follows from~\cite[Corollary 8]{Ma1303} that
\begin{equation}\label{Ma}
L(x,z)=-\sqrt{\frac{x-1}{x+1}}\left(\frac{\sqrt{x^2-1}+x\sin(z\sqrt{x^2-1})}{1-x\cos(z\sqrt{x^2-1})}\right).
\end{equation}

As an extension of~\eqref{Grammar:02}, it is natural to consider
the grammar
\begin{equation}\label{Grammar:03}
G=\{w\rightarrow w+wx, x\rightarrow x+xy, y\rightarrow y+x^2\}.
\end{equation}

From~\eqref{Grammar:03}, we have
\begin{equation*}
\begin{split}
D(w)&=w(1+x),\\
D^2(w)&=w(1+3x+xy+x^2);\\
D^3(w)&=w(1+7x+6xy+xy^2+6x^2+3x^2y+2x^3).
\end{split}
\end{equation*}

Define $$D^n(w)=w\sum_{i,j\geq 0}t_{n,i,j}x^{i}y^{j}.$$

Since
\begin{align*}
D^{n+1}(w)&=D\left(w\sum_{i,j\geq 0}t_{n,i,j}x^{i}y^{j}\right)\\
&=\sum_{i,j}(1+i+j)t_{n,i,j}x^{i}y^{j}+\sum_{i,j}t_{n,i,j}x^{i+1}y^j+
\sum_{i,j}it_{n,i,j}x^{i}y^{j+1}+
\sum_{i,j}jt_{n,i,j}x^{i+2}y^{j-1}.
\end{align*}
we get
\begin{equation}\label{eq:tnij}
t_{n+1,i,j}=(1+i+j)t_{n,i,j}+t_{n,i-1,j}+ib_{n,i,j-1}+(j+1)b_{n,i-2,j+1}
\end{equation}
for $i,j\geq 0$ , with the initial conditions $t_{0,i,j}$ to be $1$ if $(i,j)=(0,0)$ or $(i,j)=(1,0)$, and to be $0$ otherwise.
Clearly, $t_{n,0,0}=1$ for $n\geq 0$.

Define
$$T=T(x,p,q)=\sum_{n,i,j\geq0}t_{n,i,j}p^iq^j\frac{x^n}{n!}.$$

We now present the third main result of this paper.
\begin{theorem}\label{mainthm:03}
The generating function $T$ is given by
$$T(x,p,q)=e^x\sqrt{\frac{p-q}{p+q}}\frac{\sqrt{p^2-q^2}+p\sin((e^x-1)\sqrt{p^2-q^2})}{p\cos((e^x-1)\sqrt{p^2-q^2})-q}.$$
Moreover, for all $n\geq 1,i\geq 1$ and $j\geq0$,
\begin{equation}\label{Explicit:tnij}
t_{n,i,j}=\Stirling{n+1}{i+j+1}a_i(i+j).
\end{equation}
\end{theorem}
\begin{proof}
The recurrence~\eqref{eq:tnij} can be written as
\begin{equation}\label{eq:Txpq}
T_x=T+p(1+q)T_p+(p^2+q)T_q.
\end{equation}

It is routine to check that the generating function
$$\widetilde{T}=\widetilde{T}(x,p,q)=e^x\sqrt{\frac{p-q}{p+q}}\frac{\sqrt{p^2-q^2}+p\sin((e^x-1)\sqrt{p^2-q^2})}{p\cos((e^x-1)\sqrt{p^2-q^2})-q}$$
satisfies~\eqref{eq:Txpq}). Also, this generating function gives $\widetilde{T}(0,p,q)=1$ and $\widetilde{T}(x,0,q)=e^x$. Hence, $T=\widetilde{T}$.

Now let us prove that $t_{n,2i-1,j}=\Stirling{n+1}{i+j+1}a_i(i+j)$. Note that
\begin{align*}
\sum_{n,i,j\geq0}t_{n,i,j-i}p^iq^j\frac{x^n}{n!}&=\sum_{n,i,j\geq0}t_{n,i,j-i}p^iq^j\frac{x^n}{n!}=\sum_{n,j\geq0}\Stirling{n+1}{j+1}L_{j}(p)q^j\frac{x^n}{n!}\\
&=e^x\sum_{j\geq0}\frac{(e^x-1)^{j}}{(j)!}L_{j}(p)q^j=e^xL(p,q(e^x-1)),
\end{align*}
Hence,
\begin{equation*}
\sum_{n,i,j\geq0}t_{n,i,j}p^iq^j\frac{x^n}{n!}=e^xL(p/q,q(e^x-1))=T(x,p,q),
\end{equation*}
as required.
\end{proof}

Let $t_n=\sum_{i\geq1,j\geq 0}t_{n,i,j}$.
It follows from~\eqref{Explicit:tnij} that
$t_n=a_n$.
Along the same lines as the proof of Corollary~\ref{openers-descents}, we get the following.

\begin{cor}\label{openers-alt}
For all $n\geq 1,i\geq 1$ and $j\geq0$, $t_{n,i,j}$ is the number of cyclically ordered partitions of $[n+1]$ with $i+j+1$ blocks and the length of the longest alternating subsequence of
the list of openers equals $i$.
\end{cor}

\section{A product of the Stirling numbers of the second kind and binomial coefficients}\label{sec:03}
\hspace*{\parindent}

Consider the grammar
\begin{equation}\label{Grammar:04}
G=\{x\rightarrow x+x^2+xy, y\rightarrow y+y^2+xy\}.
\end{equation}

From~\eqref{Grammar:04}, we have
\begin{equation*}
\begin{split}
D(x)&=x+xy+x^2,\\
D^2(x)&=x+3xy+2xy^2+3x^2+4x^2y+2x^3,\\
D^3(x)&=x+7xy+12xy^2+6xy^3+7x^2+24x^2y+18x^2y^2+12x^3+18x^3y+6x^4.
\end{split}
\end{equation*}

For $n\geq 0$, we define $$D^n(x)=\sum_{i\geq1,j\geq 0}c_{n,i,j}x^iy^j.$$

Since
\begin{align*}
D^{n+1}(x)&=D\left(\sum_{i,j}c_{n,i,j}x^iy^j\right)\\
&=\sum_{i,j}(i+j)c_{n,i,j}x^{i}y^{j}+\sum_{i,j}(i+j)c_{n,i,j}x^{i}y^{j+1}+\sum_{i,j}(i+j)c_{n,i,j}x^{i+1}y^{j},
\end{align*}
we get
\begin{equation}\label{eq:cnij}
c_{n+1,i,j}=(i+j)c_{n,i,j}+(i+j-1)c_{n,i,j-1}+(i+j-1)c_{n,i-1,j}
\end{equation}
for $i\geq 1$ and $j\geq 0$, with the initial conditions $c_{0,i,j}$ to be $1$ if $(i,j)=(1,0)$, and to be $0$ otherwise.
Clearly, $c_{n,1,0}=1$ for $n\geq 1$.

Define
$$C=C(x,p,q)=\sum_{n,i,j\geq0}c_{n,i,j}\frac{x^{n}}{n!}p^iq^j.$$

We now present the fourth main result of this paper.
\begin{theorem}\label{mainthm:03}
The generating function $C$ is given by
$$C(x,p,q)=\frac{pe^x}{1-(p+q)(e^x-1)}.$$
Moreover, for all $n,i,j\geq1$,
\begin{equation}\label{Explicit:cnij}
c_{n,i,j}=(i+j-1)!\Stirling{n+1}{i+j}\binom{i+j-1}{j}.
\end{equation}
\end{theorem}
\begin{proof}
The recurrence~\eqref{eq:cnij} can be written as
\begin{equation}\label{eq:Cxpq}
C_x=p(1+p+q)C_p+q(1+p+q)C_q.
\end{equation}
It is routine to check that the generating function
$$\widetilde{C}(x,p,q)=\frac{pe^x}{1-(p+q)(e^x-1)}$$
satisfies (\ref{eq:Cxpq}). Also, this generating function gives $\widetilde{C}(0,p,q)=p$ and $\widetilde{C}(x,0,q)=0$. Hence, $C=\widetilde{C}$.
Now let us prove that $c_{n,i,j}=(i+j-1)!\Stirling{n+1}{i+j}\binom{i+j-1}{j}$. Note that
\begin{align*}
\frac{d}{dx}\sum_{n,i,k\geq0}c_{n,i,k+1-i}\frac{x^{n+1}}{(n+1)!}v^{i+1}w^k
&=v\frac{d}{dx}\sum_{k\geq0}\left(\sum_{n\geq k+1}\Stirling{n+1}{k+1}\frac{x^{n+1}}{(n+1)!}\sum_{i=0}^k\binom{k}{i}v^i\right)k!w^k\\
&=v\frac{d}{dx}\sum_{k\geq0}\frac{(e^x-1)^{k+1}}{k+1}(1+v)^kw^k\\
&=\frac{ve^x}{1-w(1+v)(e^x-1)},
\end{align*}
which implies
\begin{align*}
C(x,vw,w)=\frac{vwe^x}{1-w(1+v)(e^x-1)},
\end{align*}
as required.
\end{proof}

Let $c_n=\sum_{i\geq1,j\geq 0}c_{n,i,j}$.
It follows from~\eqref{Explicit:cnij} that
$c_n=\sum_{k= 0}^n2^kk!\Stirling{n+1}{k+1}$.

\section{Descent statistic of hyperoctahedral group and perfect matchings}\label{sec:04}
\hspace*{\parindent}

Let us first recall some definitions. The {\it Whitney numbers of the second kind} $W_m(n,k)$
can be explicitly defined by
$$W_m(n,k)=\sum_{i=k}^n\binom{n}{i}m^{i-k}\Stirling{i}{k}.$$
They satisfy the recurrence
$$W_m(n,k)=W_m(n-1,k-1)+(1+mk)W_m(n-1,k),$$
with initial conditions $W_m(0,0)=1$ and $W_m(n,0)=0$ for $n\geq 1$ (see~\cite{Cheon12}).
In particular, $W_2(n,k)$ also known as the type $B$ analogue of Stirling numbers of the second kind (see~\cite[A039755]{Sloane}).

The {\it hyperoctahedral group} $B_n$ is the group of signed permutations of the set $\pm[n]$ such that $\pi(-i)=-\pi(i)$ for all $i$, where $\pm[n]=\{\pm1,\pm2,\ldots,\pm n\}$. For each $\pi\in B_n$,
we define
\begin{equation*}
\des_B(\pi):=\#\{i\in\{0,1,2,\ldots,n-1\}|\pi(i)>\pi({i+1})\},
\end{equation*}
where $\pi(0)=0$.

Let
$${B}_n(x)=\sum_{\pi\in B_n}x^{\des_B(\pi)}=\sum_{k=0}^nB(n,k)x^{k}.$$
The polynomial $B_n(x)$ is called an {\it Eulerian polynomial of type $B$}, while $B(n,k)$ is called an {\it Eulerian number of type $B$} (see~\cite[A060187]{Sloane}).
The numbers $B(n,k)$ satisfy
the recurrence relation
\begin{equation*}
B(n+1,k)=(2k+1)B(n,k)+(2n-2k+3)B(n,k-1)
\end{equation*}
for $n,k\geq 0$, where $B(0,0)=1$ and $B(0,k)=0$ for $k\geq 1$.
The first few of the polynomials ${B}_n(x)$ are
\begin{align*}
  B_1(x)& =1+x, \\
  B_2(x)& =1+6x+x^2, \\
  B_3(x)& =1+23x+23x^2+x^3.
\end{align*}

Recall that a {\it perfect matching} of $[2n]$ is a partition of $[2n]$ into $n$ blocks of size $2$.
Denote by $N({n,k})$ the number of perfect matchings of $[2n]$ with the restriction that only $k$ matching pairs
have odd smaller entries (see~\cite[{A185411}]{Sloane}). It is easy to verify that
\begin{equation}\label{recurrence-11}
N({n+1,k})=2kN({n,k})+(2n-2k+3)N({n,k-1}).
\end{equation}
Let
$N_n(x)=\sum_{k=1}^nN({n,k})x^k$.
It follows from~(\ref{recurrence-11}) that
$$N_{n+1}(x)=(2n+1)xN_n(x)+2x(1-x)\frac{d}{dx}N_n(x),$$
with initial values $N_0(x)=1$, $N_1(x)=x$ and $N_2(x)=2x+x^2$.
Let
$$N(x,t)=\sum_{n\geq 0}N_n(x)\frac{t^n}{n!}.$$
It is well known that
\begin{equation}\label{eqNN1}
N(x,t)=e^{xt}\sqrt{\frac{1-x}{e^{2xt}-xe^{2t}}}.
\end{equation}

There is a grammatical characterization of the numbers $N(n,k)$ and $B(n,k)$:
if
\begin{equation}\label{Grammar:TypeB}
G=\{x\rightarrow xy^2, y\rightarrow x^2y\},
\end{equation}
then
\begin{equation*}
D^n(x)=\sum_{k=0}^nN(n,k)x^{2n-2k+1}y^{2k},~
D^n(xy)=\sum_{k=0}^nB(n,k)x^{2n-2k+1}y^{2k+1},\\
\end{equation*}
which was obtained in~\cite[Theorem 10]{Ma1302}.

As an extension of~\eqref{Grammar:TypeB}, it is natural to consider the grammar
\begin{equation}\label{Grammar:05}
G=\{x\rightarrow x+xy^2, y\rightarrow y+x^2y\}.
\end{equation}

From~\eqref{Grammar:05}, we have
\begin{equation*}
\begin{split}
D(x)&=x+xy^2,\\
D^2(x)&=x+4xy^2+xy^4+2x^3y^2;\\
D(xy)&=2xy+xy^3+x^3y,\\
D^2(xy)&=4xy+6xy^3+xy^5+6x^3y+6x^3y^3+x^5y.
\end{split}
\end{equation*}

Define $$D^n(x)=\sum_{i,j\geq 0}e_{n,i,j}x^{2i+1}y^{2j},$$
$$D^n(xy)=\sum_{i,j\geq 0}f_{n,i,j}x^{2i+1}y^{2j+1},$$

Since
\begin{align*}
D^{n+1}(x)&=D\left(\sum_{i,j\geq 0}e_{n,i,j}x^{2i+1}y^{2j}\right)\\
&=\sum_{i,j}(2i+2j+1)e_{n,i,j}x^{2i+1}y^{2j}+\sum_{i,j}(2i+1)e_{n,i,j}x^{2i+1}y^{2j+2}+
\sum_{i,j}2je_{n,i,j}x^{2i+3}y^{2j},
\end{align*}
we get
\begin{equation}\label{eq:enij}
e_{n+1,i,j}=(2i+2j+1)e_{n,i,j}+(2i+1)e_{n,i,j-1}+2je_{n,i-1,j}
\end{equation}
for $i,j\geq 0$, with the initial conditions $e_{0,0,0}=1$ and $e_{0,i,j}=0$ if $(i,j)\neq (0,0)$.

Similarly,
since
\begin{align*}
D^{n+1}(xy)&=D\left(\sum_{i,j\geq 0}f_{n,i,j}x^{2i+1}y^{2j+1}\right)\\
&=\sum_{i,j}(2i+2j+2)f_{n,i,j}x^{2i+1}y^{2j+1}+\sum_{i,j}(2i+1)f_{n,i,j}x^{2i+1}y^{2j+3}+\\
&\sum_{i,j}(2j+1)f_{n,i,j}x^{2i+3}y^{2j+1},
\end{align*}
we get
\begin{equation}\label{eq:fnij}
f_{n+1,i,j}=(2i+2j+2)f_{n,i,j}+(2i+1)f_{n,i,j-1}+(2j+1)f_{n,i-1,j}
\end{equation}
for $i,j\geq 0$, with the initial conditions $f_{0,0,0}=1$ and $f_{0,i,j}=0$ if $(i,j)\neq (0,0)$.

Define
$$E=E(x,p,q)=\sum_{n,i,j\geq0}e_{n,i,j}p^iq^j\frac{x^n}{n!},$$
$$F=F(x,p,q)=\sum_{n,i,j\geq0}f_{n,i,j}p^iq^j\frac{x^n}{n!}.$$

We now present the fifth main result of this paper.
\begin{theorem}\label{mainthm:05}
The generating functions $E$ and $F$ are respectively given by
$$E(x,p,q)=e^{x+q(e^{2x}-1)/2}\sqrt{\frac{p-q}{pe^{q(e^{2x}-1)}-qe^{p(e^{2x}-1)}}}$$
and
$$F(x,p,q)=\frac{(p-q)e^{(p-q)(e^{2x}-1)/2+2x}}{p-qe^{(p-q)(e^{2x}-1)}}.$$
Moreover, for all $n,i,j\geq1$,
$$e_{n,i,j}=W_2(n,i+j)N(i+j,j),$$
$$f_{n,i,j}=2^{n-i-j}\Stirling{n+1}{i+j+1}B(i+j,j).$$
\end{theorem}
\begin{proof}
By rewriting \eqref{eq:enij} and \eqref{eq:fnij} in terms of generating functions, we obtain
\begin{equation}\label{eq:Expq}
E_x=2p(1+q)E_p+2q(1+p)E_q+(1+q)E,
\end{equation}
and
\begin{equation}\label{eq:Fxpq}
F_x=2p(1+q)F_p+2q(1+p)F_q+(2+p+q)F.
\end{equation}
It is routine to check that the generating functions
$$\widetilde{E}=\widetilde{E}(x,p,q)=e^{x+q(e^{2x}-1)/2}\sqrt{\frac{p-q}{pe^{q(e^{2x}-1)}-qe^{p(e^{2x}-1)}}}$$
and
$$\widetilde{F}=\widetilde{F}(x,p,q)=\frac{(p-q)e^{(p-q)(e^{2x}-1)/2+2x}}{p-qe^{(p-q)(e^{2x}-1)}}$$
satisfies (\ref{eq:Expq}) and (\ref{eq:Fxpq}), respectively. Also, this generating functions give $\widetilde{F}(0,p,q)=\widetilde{E}(0,p,q)=1$, $\widetilde{E}(x,0,q)=\widetilde{F}(x,0,q)=e^{q(2^{2x}-1)/2+2x}$, $\widetilde{E}(x,p,0)=e^x$ and $\widetilde{F}(x,p,0)=e^{p(2^{2x}-1)/2+2x}$. Hence, $E=\widetilde{E}$ and $F=\widetilde{F}$.

Now let us prove that the generating function for the sequences $e_{n,i-j,j}=W_2(n,i)N(i,j)$ and $f_{n,i-j,j}=2^{n-i}\Stirling{n+1}{i+1}B(i,j)$ are given by $E(x,p,pq)$ and $F(x,p,pq)$. By (\ref{eqNN1}) and \cite[A039755]{Sloane}, we have
\begin{align*}
\sum_{n,i,k\geq0}e_{n,i-j,j}\frac{x^n}{n!}p^iq^j
&=\sum_{i\geq0}\left(\sum_{n\geq i}W_2(n,i)\frac{x^n}{n!}\right)N_i(q)\\
&=e^{x}\sum_{i\geq0}\frac{p^i(e^x-1)^{i}}{2^ii!}N_i(q)\\
&=E(x,p,pq),
\end{align*}
and
\begin{align*}
\frac{d}{dx}\sum_{n,i,k\geq0}f_{n,i-j,j}\frac{x^{n+1}}{(n+1)!}p^iq^j
&=\frac{d}{dx}\sum_{i\geq0}\left(\sum_{n\geq i}\Stirling{n+1}{i+1}\frac{(2x)^{n+1}p^i}{2^{i+1}(n+1)!}\right)B_i(q)\\
&=e^{2x}\sum_{i\geq0}\frac{p^i(e^x-1)^{i}}{2^ii!}B_i(q)\\
&=F(x,p,pq),
\end{align*}
which completes the proof.
\end{proof}

\section{Concluding remarks}\label{sec:05}
\hspace*{\parindent}

In this paper, we explore some combinatorial structures associated with~\eqref{Lotka-Volterra-model-2}. In fact, there are many other extension of~\eqref{Lotka-Volterra-model}. For example, many authors investigated the following generalized Lotka-Volterra system (see~\cite{Ollagnier01}):
\begin{equation*}\label{L-V-2}
\frac{dx}{dt}=x(Cy+z),
\frac{dy}{dt}=y(Az+x),
\frac{dz}{dt}=z(Bx+y).
\end{equation*}

Consider the grammar
\begin{equation*}\label{Grammar:07}
G=\{x\rightarrow x(y+z), y\rightarrow y(z+x), z\rightarrow z(x+y)\}.
\end{equation*}

Define $$D^n(x)=\sum_{i\geq1,j\geq 0}g_{n,i,j}x^iy^jz^{n+1-i-j}.$$

By induction, one can easily verify the the following: for all $n\geq 1,i\geq1$ and $j\geq0$, we have
\begin{equation*}
g_{n,i,0}=\Eulerian{n}{i},
g_{n,i,n+1-i}=\Eulerian{n}{i},
g_{n,1,j}=\Eulerian{n+1}{j+1}.
\end{equation*}



\begin{thebibliography}{22}
\bibitem{Aguiar04}
M. Aguiar, N. Bergeron, K. Nyman, The peak algebra and the descent algebras of types B and D,
\newblock{\em Trans. Amer. Math. Soc.} 356 (2004), no. 7, 2781--2824.


\bibitem{Allen07}
L.J.S. Allen, An Introduction to Mathematical Biology, Prentice Hall, New Jersey, 2007.


\bibitem{Chauvet02}
E. Chauvet, J.E. Paullet, J.P. Previte, Z. Walls,
A Lotka-Volterra Three-species Food Chain, \newblock {\em Math. Mag.} 75 (2002) 243--255.

\bibitem{Chen93}
W.Y.C. Chen, Context-free grammars, differential operators and formal
power series, \newblock {\em Theoret. Comput. Sci.} 117 (1993), 113--129.

\bibitem{Cheon12}
G.-S. Cheon, J.-H. Jung, The $r$-Whitney numbers of Dowling lattices, \newblock {\em Discrete Math.} 312 (2012), 2337--2348.

\bibitem{Dilks09}
K. Dilks, T.K. Petersen, J.R. Stembridge, Affine descents and the Steinberg torus, \newblock {\em Adv. in
Appl. Math.} 42 (2009) 423--444.


\bibitem{Dumont96}
D. Dumont, Grammaires de William Chen et d\'erivations dans les arbres et
arborescences, \newblock {\em S\'em. Lothar. Combin.} 37, Art. B37a (1996) 1--21.

\bibitem{Evans99}
C.M. Evans, G.L. Findley, Analytic solutions to a family of Lotka-Volterra related differential equations, \newblock {\em J. Math. Chem.} 25 (1999) 181--189.

\bibitem{Francon79} J. Fran\c{c}on and G. Viennot, Permutations selon leurs pics, creux, doubles mont\'{e}es et double descentes, nombres d'Euler et nombres de Genocchi, \newblock{\em Discrete Math.} 28 (1979) 21--35.

\bibitem{Franssens07}
G. R. Franssens, Functions with derivatives given by polynomials in the function
itself or a related function, \newblock {\em Anal. Math.} 33 (2007) 17--36.

\bibitem{Hoffman99}
M. E. Hoffman, Derivative polynomials, Euler polynomials, and associated integer
sequences, \newblock {\em Electron. J. Combin.} 6 (1999) \#R21.

\bibitem{Ma121}
S.-M. Ma, Derivative polynomials and enumeration of permutations by number of interior and left peaks,
\newblock {\em Discrete Math.} 312 (2012) 405--412.

\bibitem{Ma13}
S.-M. Ma, A family of two-variable derivative polynomials for tangent and secant,
\newblock {\em Electron. J. Combin.} 20(1) (2013) \#P11.

\bibitem{Ma1302}
S.-M. Ma, Some combinatorial arrays generated by context-free grammars, \newblock{\em European J. Combin.} 34 (2013) 1081--1091.

\bibitem{Ma1303}
S.-M. Ma, Enumeration of permutations by number of alternating runs, \newblock{\em Discrete Math.} 313 (2013) 1816--1822.

\bibitem{MiMura78}
M. MiMura, T. Nishida, On a Certain Semilinear Parabolic System Related to the Lotka-Volterra Ecological Model,
\newblock{\em Publ. RIMS, Kyoto Univ.} 14 (1978) 269--282.


\bibitem{Ollagnier01}
J.M. Ollagnier, Liouvillian integration of the Lotka-Volterra system, \newblock{\em Qual. Theory Dyn. Syst.} 2 (2001) 307--358.


\bibitem{Shih97}
S.-D. Shih, The period of a Lotka-Volterra system, \newblock {\em Taiwanese J. Math.} 1(4) (1997) 451--470.

\bibitem{Singer92}
M. F. Singer, Liouvillian first integrals of differential equations, \newblock {\em Trans. Amer.
Math. Soc.} 333 (1992) 673--688.

\bibitem{Sloane}
N.J.A. Sloane, The On-Line Encyclopedia of Integer Sequences,
published electronically at
http://oeis.org, 2010.

\bibitem{Sta07}
R.P. Stanley,
Increasing and decreasing subsequences and their variants, Proc.
Internat. Cong. Math (Madrid 2006), American Mathematical Society, 549--579,
2007.


\bibitem{Sta08}
R.P. Stanley, Longest alternating subsequences of permutations,
\newblock {\em Michigan Math. J.} 57 (2008) 675--687.
\end{thebibliography}
\end{document}